\documentclass[a4paper]{article}

\usepackage{amssymb,latexsym,amsmath,epsfig,amsthm,fullpage,authblk,titlesec,secdot} 
\usepackage{filecontents}

\makeatletter

\newcommand{\Ein}{\text{\rm Ein}}

\makeatother

\newtheorem{theorem}{Theorem}
\newtheorem*{theorem*}{Theorem}
\newtheorem{lemma}{Lemma}

\newtheorem{corollary}{Corollary}

\newtheorem*{remark*}{Remark}

\begin{document}
\title{\uppercase{A Family of Multiple Integrals Connected with Relatives of the Dickman Function}}
\author{C. S. Franze \thanks{franze.3@osu.edu}}
\affil{Department of Mathematics,\\ The Ohio State University}
\date{}
\maketitle
\abstract{In this paper we present an asymptotic expansion for a family of multiple integrals connected with relatives of the Dickman function. The coefficients of this expansion have a similar arithmetic structure as those appearing in Soundararajan's work on an analogous expansion for the Dickman function.}
\section{Introduction}
Let $d_\kappa(n)$ be the number of representations of $n$ as a product of exactly $\kappa$ positive integers, and $P(n)$ denote the greatest prime factor of $n$. Setting $u=\log x/\log y$, de Bruijn and Van Lint \cite{DeBruijn} showed that as $x,y\rightarrow\infty$ with $u$ bounded,
\[
    \sum_{\substack{n\le x\\P(n)\le y}}d_\kappa(n)\sim \rho_\kappa(u)\ x\log^{\kappa-1}y,
\]
where $\rho_\kappa(u)=u^{\kappa-1}/(\kappa-1)!$ when $0<u\le1$, $\rho_\kappa(u)=0$ when $u\le0$, and
\begin{equation}\label{E:altdiffeq}
  \left(u^{1-\kappa} \rho_{\kappa}(u)\right)'=-\kappa u^{-\kappa}\rho_{\kappa}(u-1),\quad u>1.
\end{equation}
The function $\rho_{\kappa}(u)$ is well-understood, and has been studied by many authors (e.g. see \cite{Hildebrand} and \cite{Smida}). For example, it is possible to write $\rho_{\kappa}(u)$ as a sum of multiple integrals, which can then be computed numerically. 
\begin{lemma}\label{decomplemma}
If $\rho_{\kappa}(u)$ is defined as above in \eqref{E:altdiffeq}, then we may write
\begin{equation*}
    \rho_\kappa(u)=\sum_{\ell=0}^{\infty}\frac{(-\kappa)^\ell}{(\kappa-1)!}K_\ell(u,\kappa-1),
\end{equation*}
where $K_{\ell}(u,\kappa):=0$ when $\ell\ge u$, $K_{0}(u,\kappa):=u^{\kappa}$, and, for $\ell\ge1$,
  \begin{equation}\label{E:multK2}
  K_{\ell}(u,\kappa):=\frac{1}{\ell!}\idotsint\limits_{\substack{t_{1},\ldots,t_{\ell}\ge1\\t_{1}+\cdots+t_{\ell}\le u}}\left(u-(t_{1}+\cdots+t_{\ell})\right)^{\kappa}\frac{dt_{1}}{t_{1}}\cdots\frac{dt_{\ell}}{t_{\ell}}.
\end{equation}
\end{lemma}
The case $\kappa=1$ corresponds to Dickman's function, but appears even earlier in Ramanujan's unpublished papers. For general $\kappa$, one can deduce Lemma \ref{decomplemma} from work of Wheeler \cite[p.498]{Wheeler2}. A proof will be supplied shortly.\\
Though there are many numerical methods for computing $K_\ell(u,\kappa)$ [e.g. see \cite{Franze}, \cite{Grupp}, \cite{Wheeler}], the asymptotic nature of $K_{\ell}(u,\kappa)$ has not been fully explored. This is somewhat surprising given also that asymptotic expansions for $\rho_{\kappa}(u)$ are available. Recently, however, an asymptotic formula for $K_{\ell}(u):=K_{\ell}(u,0)$ was developed by Soundararajan \cite[see Theorem 1, and Propositon 1]{Sound}, in connection with Lemma \ref{decomplemma}.
\begin{theorem}[Soundararajan, 2012]\label{ST1}
For each $\ell\ge1$, provided $u$ is sufficiently large,
\begin{equation*}
  K_{\ell}(u)=\sum_{r=0}^{\ell}\frac{(-1)^r}{(\ell-r)!}C_r \log^{\ell-r}u+O_{\ell}\left(\frac{\log^{\ell}u}{u}\right),
\end{equation*}
where the constants $C_r$ are generated by
\begin{equation}\label{E:Soundgen}
  \sum_{r=0}^{\infty}C_{r}z^r=\frac{e^{\gamma z}}{\Gamma(1-z)}.
\end{equation}
\end{theorem}
The formula was conjectured by Broadhurst \cite{Broadhurst} in the course of his investigations into a generalized class of polylogarithms. In this paper, we generalize Theorem \ref{ST1} to all integer $\kappa\ge 0$. Our formula will be uniform in $u\ge\ell$, where $\ell\ge 1$. Specifically, we prove
\begin{theorem}\label{T_1}
For each integer $\kappa\ge0$, and $\ell\ge1$, provided $u\ge\ell$,
\begin{equation}\label{KMain}
K_{\ell}(u,\kappa)=\sum_{m=0}^{\min\left(\kappa,\ell\right)}\sum_{n=m}^{\kappa}\sum_{r=0}^{\ell-m} D_{m,n,r,\kappa,\ell}\ u^{\kappa-n}\log^{\ell-m-r}u+O_{\kappa,\ell}\left(\frac{\log^{\ell} eu}{u}\right),
\end{equation}
where
\begin{equation}\label{DDef}
 D_{m,n,r,\kappa,\ell}=\frac{(-1)^r \Gamma(\kappa+1)E_{n,m}}{m!(\ell-m-r)!}C_{r,\kappa-n},
\end{equation}
and the constants $C_{r,\kappa}$ and $E_{n,m}$ are generated by
\begin{equation}\label{CDef}
  \sum_{r=0}^{\infty}C_{r,\kappa}\ z^{r}=\frac{e^{\gamma z}}{\Gamma\left(\kappa+1-z\right)},
\end{equation}
and
\begin{equation}\label{EDef}
  \sum_{n=0}^{\infty}E_{n,m}\ z^{n}=\Ein(z)^m.
\end{equation}
\end{theorem}
The function $\Ein(z)$ appearing above is defined by an integral,
\begin{equation*}
   \Ein(z):=\int_{0}^{z}\frac{1-e^{-t}}{t}dt.
\end{equation*}
Using this representation, it is not hard to show that the coefficients $E_{n,m}$ appearing in \eqref{EDef} satisfy
\begin{equation}\label{Ealtexp}
  n!E_{n,m}=\sum_{\substack{n_1,\ldots,n_m\ge1\\n_1+\cdots+n_m=n}}\frac{(-1)^{n+m}}{n_1\cdots n_m}\binom{n}{n_1,\ldots,n_m},
\end{equation}
where $E_{n,m}=0$ if $0\le n\le m-1$, and $m\ge1$. If $m=0$, $E_{0,0}=1$ and $E_{n,0}=0$ for $n>0$.
\section{An Integral Decomposition}
We now turn our attention to the proof of Lemma \ref{decomplemma}. As remarked earlier, it can be deduced from more general work of Wheeler. As a special case of that work, we have Lemma 2 below. The proof is short and so we take the opportunity to reproduce it here.
\begin{lemma}[Wheeler, 1990]\label{WheelerSP}
If $\rho_{\kappa}(u)$ is defined as above in \eqref{E:altdiffeq}, then we may write
\begin{equation}\label{E:rhorep1}
  \rho_{\kappa}(u)=\sum_{0\le \ell<u}\frac{(-\kappa)^{\ell}}{(\kappa-1)!}K_\ell(u,\kappa-1),
\end{equation}
where $K_{\ell}(u,\kappa):=0$ when $\ell\ge u$, $K_{\ell}(u,\kappa)=u^{\kappa}$, and for $\ell\ge1$,
\begin{equation}\label{E:FirstKrep}
  K_{\ell}(u,\kappa)=u^{\kappa}\int_{\ell}^{u}t^{-\kappa-1}K_{\ell-1}(t-1,\kappa)dt.
\end{equation}
\end{lemma}
\begin{proof}
We must show that the expression on the right-hand side of \eqref{E:rhorep1} satisfies the same delay differential equation as $\rho_{\kappa}(u)$, namely \eqref{E:altdiffeq}. To this end, note that from the definition of $K_\ell(u,\kappa)$ in \eqref{E:FirstKrep},
\begin{equation}\label{UsefulKdiff}
  \left(u^{-\kappa}K_{\ell}(u,\kappa)\right)'=u^{-\kappa-1}K_{\ell-1}(u-1,\kappa).
\end{equation}
Recalling the right-hand side of \eqref{E:rhorep1} and using \eqref{UsefulKdiff} then gives
\begin{align*}
  \left(u^{1-\kappa}\sum_{0\le\ell<u}\frac{(-\kappa)^{\ell}}{(\kappa-1)!}K_{\ell}(u,\kappa-1)\right)'
    &=-\kappa u^{-\kappa}\sum_{1\le\ell<u}\frac{(-\kappa)^{\ell-1}}{(\kappa-1)!}K_{\ell-1}(u-1,\kappa-1).
\end{align*}
Re-indexing this last sum to $0\le\ell<u-1$ then completes the proof. 
\end{proof}
To deduce Lemma \ref{decomplemma} from Lemma \ref{WheelerSP}, we must show that the integrals arising from \eqref{E:FirstKrep} are equivalent to those appearing in \eqref{E:multK2}. Iterating the integral relation in \eqref{E:FirstKrep}, we find that
\begin{equation}\label{E:multK1}
  K_{\ell}(u,\kappa)=u^\kappa\int_{\ell}^{u}\int_{\ell-1}^{t_{\ell}-1}\cdots\int_{1}^{t_2-1}\prod_{i=1}^{\ell}\left(t_i-1\right)^\kappa\frac{dt_1}{t_1^{\kappa+1}}\cdots\frac{dt_{\ell-1}}{t_{\ell-1}^{\kappa+1}}\frac{dt_\ell}{t_\ell^{\kappa+1}}.
\end{equation}
A change of variables is now required to show that \eqref{E:multK1} and \eqref{E:multK2} are equivalent. The details are outlined below.
\begin{proof}[Proof of Lemma 1]
Observe that upon ordering the variables, the right-hand side of \eqref{E:multK2} becomes
\begin{equation*}
  \idotsint\limits_{\substack{1\le t_\ell\le\cdots\le t_1\\t_1+\cdots+t_\ell\le u}}\left(u-(t_1+\cdots+t_\ell)\right)^{\kappa}\frac{dt_1}{t_1}\cdots\frac{dt_\ell}{t_\ell},
\end{equation*}
or equivalently,
\begin{equation}\label{E:OrderedInt2} 
  \int_{1}^{\frac{u}{\ell}}\int_{t_\ell}^{\frac{u-t_\ell}{\ell-1}}\cdots\int_{t_2}^{\frac{u-(t_2+\cdots+t_\ell)}{1}}\left(u-\sum_{i=1}^{\ell} t_i\right)^\kappa \frac{dt_{1}}{t_{1}}\cdots\frac{dt_{\ell-1}}{t_{\ell-1}}\frac{dt_\ell}{t_\ell}.
\end{equation}
Next, we make the change of variables
\begin{equation}\label{E:Change_of_Variables}
  t_i=\frac{u}{v_i}\prod_{j=i+1}^{\ell}\left(1-\frac{1}{v_j}\right), \quad\text{for $1\le i\le \ell$},
\end{equation}
adopting the convention that the empty product is one. One can verify that
\begin{equation}\label{E:New_Integrand}
  \prod_{i=1}^{\ell}\left(1-\frac{1}{v_i}\right)=\frac{1}{u}\left(u-\sum_{i=1}^{\ell}t_i\right),
\end{equation}
and that the limits of integration appearing in \eqref{E:OrderedInt2} translate to those in \eqref{E:multK1}. Moreover, the Jacobian of the transformation can easily be computed using \eqref{E:Change_of_Variables},
\begin{equation}\label{E:Jacobian}
  \left|\frac{\partial(t_1,\ldots,t_\ell)}{\partial(v_1,\ldots,v_\ell)}\right|=\frac{t_1\cdots t_\ell}{v_1\cdots v_\ell},
\end{equation}
since $\frac{\partial t_i}{\partial v_i}=-\frac{t_i}{v_i}$, and $\frac{\partial t_i}{\partial v_j}=0$ for $j<i$. Therefore, after collecting \eqref{E:Change_of_Variables}, \eqref{E:New_Integrand}, and \eqref{E:Jacobian}, the integral in \eqref{E:OrderedInt2} takes the shape
\begin{equation*}
    \int_{\ell}^{u}\int_{\ell-1}^{v_{\ell}-1}\cdots\int_{1}^{v_2-1}\left(u\prod_{i=1}^{\ell}\left(1-\frac{1}{v_i}\right)\right)^\kappa\left|\frac{\partial(t_1,\ldots,t_\ell)}{\partial(v_1,\ldots,v_\ell)}\right|\frac{dv_1}{t_1}\cdots\frac{dv_{\ell-1}}{t_{\ell-1}}\frac{dv_\ell}{t_\ell},
\end{equation*}
or equivalently,
\begin{equation*}
    u^\kappa\int_{\ell}^{u}\int_{\ell-1}^{v_{\ell}-1}\cdots\int_{1}^{v_2-1}\prod_{i=1}^{\ell}\left(1-\frac{1}{v_i}\right)^\kappa\frac{dv_1}{v_1}\cdots\frac{dv_{\ell-1}}{v_{\ell-1}}\frac{dv_\ell}{v_\ell}.
\end{equation*}
This last integral is clearly equivalent to \eqref{E:multK1}, and our proof of Lemma \ref{decomplemma} is complete.
\end{proof}
We pause now to record a useful relationship that follows readily from the particular form of $K_{\ell}(u,\kappa)$ appearing in Lemma \ref{decomplemma}. In particular, using \eqref{E:multK2}, it is easy to relate $K_{\ell}(u,\kappa)$ back to $K_{\ell}(u,0)$, as
\begin{equation}\label{Kconvolvidentity}
  K_{\ell}(u,\kappa)=\kappa\int_{\ell}^{u}(u-t)^{\kappa-1}K_{\ell}(t,0)dt.
\end{equation}
To see this, observe that the right-hand side of \eqref{E:multK2} may be rewritten using symmetry as
\begin{equation*}
  \frac{\kappa!}{\ell!}\idotsint\limits_{\substack{t_1,\ldots,t_\ell\ge1 \\ t_1+\cdots+t_\ell\le u}}\int_{t_1+\cdots+t_\ell}^{u}\int_{t_1+\cdots+t_\ell}^{s_\kappa}\cdots \int_{t_1+\cdots+t_\ell}^{s_{2}}ds_{1}\cdots ds_{\kappa-1}ds_{\kappa}\frac{dt_1}{t_1}\cdots \frac{dt_\ell}{t_\ell}.
\end{equation*}
Interchanging the order of integration, this becomes
\begin{equation*}
  \kappa!\int_{\ell}^{u}\int_{\ell}^{s_\kappa}\cdots \int_{\ell}^{s_{2}}K_{\ell}(s_{1},0)ds_{1}\cdots ds_{\kappa-1}ds_{\kappa}.
\end{equation*}
Cauchy's repeated integration formula then gives the expression appearing on the right-hand side of \eqref{Kconvolvidentity}.
\section{Generalized Dickman Constants}
Having proved Lemma \ref{decomplemma}, we move on and address the constants $C_{r}$ appearing in Theorem \ref{ST1}, and $C_{r,\kappa}$ in Theorem \ref{T_1}. These constants are initially defined by complex integrals. Broadhurst \cite{Broadhurst} referred to the constants $C_r$ as the \emph{Dickman constants}, and conjectured their generating function, $e^{\gamma z}/\Gamma(1-z)$. Soundararajan \cite[Proposition 1]{Sound} later proved this conjecture.
\begin{lemma}[Soundararajan, 2012]\label{SoundCr}
For natural numbers $r\ge0$, define the constants
\begin{equation}
C_{r}:=\frac{1}{r!}\frac{1}{2\pi i}\int_{c-i\infty}^{c+i\infty}\frac{e^{s}}{s}\left(\log s+\gamma\right)^{r}ds,
\end{equation}
where $c>0$ and the integral is interpreted as the Cauchy principal value. We have, for all complex $z$,
\begin{equation*}
\sum_{r=0}^{\infty}C_{r}z^{r}=\frac{e^{\gamma z}}{\Gamma\left(1-z\right)}.
\end{equation*}
\end{lemma}
\begin{proof}
See \cite[Proposition 1]{Sound}.
\end{proof}
The generating function connects $C_r$ to the values of the Riemann zeta-function at integers. This connection is made explicit by Corollary \ref{CorCr} below.
\begin{corollary}\label{CorCr}
  The constants $C_r$ defined above in Lemma \ref{SoundCr} can also be written as
  \begin{equation*}
  C_r=\frac{1}{r!}\sum_{k=1}^{r}(-1)^{k}B_{r,k}\left(0,1!\zeta(2),2!\zeta(3),\ldots,(r-k)!\zeta(r-k+1)\right),\quad r\ge1,
  \end{equation*}
  where $C_0=1$, and $B_{r,k}$ denotes the Bell polynomial,
\begin{equation*}
  B_{r,k}(x_1,\ldots,x_{r-k+1}):=\sum_{\substack{j_1,\ldots,j_{r-k+1}\ge0\\ j_1+j_2+\cdots+j_{r-k+1}=k\\ j_1+2j_2+\cdots+(r-k+1)j_{r-k+1}=r}}\frac{r!}{j_1!\cdots j_{r-k+1}!}\prod_{i=1}^{r-k+1}\left(\frac{x_i}{i!}\right)^{j_i}.
\end{equation*}
\end{corollary}
\begin{proof}[Proof of Corollary 1]
Consider the power series expansion,
\begin{equation*}
  \log\Gamma(1+z)=-\gamma z+\sum_{k=2}^{\infty} \frac{\zeta(k)}{k}(-z)^k,\quad |z|<1.
\end{equation*}
Replacing $z$ with $-z$, and then exponentiating gives
\begin{equation}\label{altgenfunc}
  \frac{e^{\gamma z}}{\Gamma(1-z)}=\exp\left(-\sum_{k=2}^{\infty} \frac{\zeta(k)}{k} z^k\right).
\end{equation}
Making use of Fa\`a di Bruno's formula to expand the right-hand side of \eqref{altgenfunc} as a power series yields
\begin{equation}\label{E:Bellgen}
  \frac{e^{\gamma z}}{\Gamma(1-z)}=1+\sum_{r=1}^{\infty}\left(\frac{1}{r!}\sum_{k=1}^{r}B_{r,k}\left(L'(0),L''(0),\ldots,L^{(r-k+1)}(0)\right)\right)z^r,
\end{equation}
where we have set $L(z)=-\sum_{k=2}^{\infty} \frac{\zeta(k)}{k} z^k$. Now, since $L^{(j)}(0)=-(j-1)!\zeta(j)$ for $j>1$, $L'(0)=0$, and $B_{r,k}(-x_1,\ldots,-x_{r-k+1})=(-1)^{k}B_{r,k}(x_1,\ldots,x_{r-k+1})$, setting \eqref{E:Soundgen} and \eqref{E:Bellgen} equal to eachother completes the proof.
\end{proof}
We will need a generalization of the Dickman constants for the purpose of this paper. Thus, the constants $C_{r,\kappa}$ defined below contain the Dickman constants $C_{r,0}=C_{r}$ as a special case.
\begin{lemma}\label{T1_begin}
For natural numbers $r,\kappa\ge0$, define the constants
\begin{equation}\label{E1_begin}
C_{r,\kappa}=\frac{1}{r!}\frac{1}{2\pi i}\int_{c-i\infty}^{c+i\infty}\frac{e^{s}}{s^{\kappa+1}}\left(\log s+\gamma\right)^{r}ds
\end{equation}
where $c>0$ and the integral is interpreted as the Cauchy principal value. We have, for all complex $z$,
\begin{equation}\label{E2_begin}
\sum_{r=0}^{\infty}C_{r,\kappa}\ z^{r}=\frac{e^{\gamma z}}{\Gamma\left(\kappa+1-z\right)}.
\end{equation}
\end{lemma}
\begin{proof}
We will assume $\kappa\ge1$ in view of Lemma \ref{SoundCr}. In this case, the integral appearing on the right side of \eqref{E1_begin} is absolutely convergent. In fact, choosing $c=1$ and then letting $s=1+it$, we find that
\begin{align*}
  |C_{r,\kappa}|&\ll\frac{M^r}{r!}\int_{-\infty}^{\infty}\frac{\left(1+\log\sqrt{1+t^2}\right)^r}{\left(\sqrt{1+t^2}\right)^{\kappa+1}}dt\ll_{\kappa}\frac{M^r}{r!}\int_{0}^{\infty}\frac{\left(1+\log \left(1+t\right)\right)^r}{\left(1+t\right)^{\kappa+1}}dt,
\end{align*}
for some constant $M>0$.\\ 
Next, we make a change of variable $w=\kappa(1+\log(1+t))$, and observe that
\begin{equation*}
  \int_{0}^{\infty}\frac{\left(1+\log \left(1+t\right)\right)^r}{\left(1+t\right)^{\kappa+1}}dt\ll_{\kappa}\frac{1}{\kappa^{r+1}}\int_{\kappa}^{\infty}w^{r}e^{-w}dw\ll_{\kappa}\frac{\Gamma(r+1)}{\kappa^{r+1}},
\end{equation*}
which gives the bound,
\begin{align*}
    |C_{r,\kappa}|\ll_{\kappa}\frac{M^{r}}{\kappa^{r+1}}.
\end{align*}
Thus, the series in \eqref{E2_begin} converges absolutely for $|z|$ inside a disk of radius $O(\kappa)$ and defines an analytic function in that region. Inside this region,
\begin{equation*}
\sum_{r=0}^{\infty}C_{r,\kappa}\ z^{r}=\frac{1}{2\pi i}\int_{c-i\infty}^{c+i\infty}\frac{e^s}{s^{\kappa+1}}e^{\gamma z}s^zds=\frac{e^{\gamma z}}{\Gamma\left(\kappa+1-z\right)}.
\end{equation*}
Here, as in Soundararajan's proof of Lemma \ref{SoundCr}, we made use of Hankel's contour integral for the reciprocal of the $\Gamma$-function. To complete the proof, observe that since $e^{\gamma z}/\Gamma\left(\kappa+1-z\right)$ is analytic for all $z\in\mathbb{C}$, the absolute convergence of the series in \eqref{E2_begin} for all $z$ follows by analytic continuation.
\end{proof}
For integers $\kappa\ge1$, the constants $C_{r,\kappa}$ can be related back to $C_{r,0}=C_{r}$ using the recursive formula,
\begin{equation*}
  C_{r,\kappa}=\sum_{j=0}^{r}\frac{C_{j,\kappa-1}}{\kappa^{r-j+1}}.
\end{equation*}
This identity follows easily from \eqref{E2_begin} since, for $|z|<\kappa$,
\begin{align*}
  \sum_{r=0}^{\infty}C_{r,\kappa}z^r=\frac{1}{\kappa-z}\frac{e^{\gamma z}}{\Gamma(\kappa-z)}=\sum_{i=0}^{\infty}\frac{z^i}{\kappa^{i+1}}\sum_{j=0}^{\infty}C_{j,\kappa-1}z^j.
\end{align*}
Multiplying the series on the right and comparing coefficients yields the recursive formula.
\section{Proof of Theorem \ref{T_1}}
We are now ready to prove Theorem \ref{T_1}. Recall that for any $c>0$,
\begin{equation*}
\frac{\Gamma(\kappa+1)}{2\pi i}\int_{c-i\infty}^{c+i\infty}e^{\lambda s}\frac{ds}{s^{\kappa+1}}=
  \begin{cases}
    \lambda^{\kappa}& \text{if $\lambda>0$,}\\
    0& \text{if $\lambda<0$.}
  \end{cases}
\end{equation*}
Using this integral to detect the condition $t_1+\cdots+t_\ell\le u$ in \eqref{E:multK2}, and then interchanging the order of integration gives
\begin{equation}\label{E11}
  K_{\ell}(u,\kappa)=\frac{\Gamma(\kappa+1)}{2\pi i\ \ell!}\int_{c-i\infty}^{c+i\infty}e^{us}E_1(s)^{\ell}\frac{ds}{s^{\kappa+1}},
\end{equation}
where
\begin{equation*}
  E_{1}(s):=\int_{1}^{\infty}\frac{e^{-ts}}{t}dt.
\end{equation*}
Now, using the relationship derived in \cite[p.28]{Sound}, namely
\begin{equation*}
\Ein(s)=\gamma+\log s+E_{1}(s),
\end{equation*}
equation \eqref{E11} becomes
\begin{equation*}
K_{\ell}(u,\kappa)=\frac{\Gamma(\kappa+1)}{\ell!}\frac{1}{2\pi i}\int_{c-i\infty}^{c+i\infty}e^{us}\left(\Ein(s)-\log s -\gamma\right)^{\ell}\frac{ds}{s^{\kappa+1}}.
\end{equation*}
Replacing $s$ with $s/u$, we have
\begin{equation}\label{E13}
K_{\ell}(u,\kappa)=\frac{\Gamma(\kappa+1)}{\ell!}u^{\kappa}\frac{1}{2\pi i}\int_{c-i\infty}^{c+i\infty}e^{s}\left(\log u-\log s-\gamma+G(u,s)\right)^{\ell}\frac{ds}{s^{\kappa+1}},
\end{equation}
where
\begin{equation}\label{Gdef}
G(u,s):=\Ein\left(\frac{s}{u}\right)=\int_{0}^{1/u}\frac{1-e^{-ts}}{t}dt.
\end{equation}
Now, using the binomial theorem, equation \eqref{E13} becomes
\begin{equation}\label{E14}
K_{\ell}(u,\kappa)=\sum_{m=0}^{\ell}\binom{\ell}{m}\frac{\Gamma(\kappa+1)}{\ell!}\frac{u^{\kappa}}{2\pi i}\int_{(c)}e^{s}G(u,s)^{m}\left(\log u-\log s-\gamma\right)^{\ell-m}\frac{ds}{s^{\kappa+1}},
\end{equation}
where we have abbreviated $\int_{c-i\infty}^{c+i\infty}=\int_{(c)}$. The main contribution in \eqref{E14} comes from the terms corresponding to integers $0\le m\le \min\left(\kappa,\ell\right)$. Now, if $\kappa<\ell$, then we use Lemma \ref{Error1} below to show that the discarded terms,
\begin{equation*}
\sum_{m=\kappa+1}^{\ell}\binom{\ell}{m}\frac{\Gamma(\kappa+1)}{\ell!}\frac{u^{\kappa}}{2\pi i}\int_{(c)}e^{s}G(u,s)^{m}\left(\log u-\log s-\gamma\right)^{\ell-m}\frac{ds}{s^{\kappa+1}},
\end{equation*}
are bounded uniformly in $u\ge\ell$ by the stated error term in \eqref{KMain}. For the remaining terms,
\begin{equation*}
\sum_{m=0}^{\min\left(\kappa,\ell\right)}\binom{\ell}{m}\frac{\Gamma(\kappa+1)}{\ell!}\frac{u^{\kappa}}{2\pi i}\int_{(c)}e^{s}G(u,s)^{m}\left(\log u-\log s-\gamma\right)^{\ell-m}\frac{ds}{s^{\kappa+1}},
\end{equation*}
we use the power series expansion in \eqref{EDef},
\begin{equation}\label{EGseries}
G(u,s)^{m}=\Ein\left(\frac{s}{u}\right)^m=\sum_{n=m}^{\infty}E_{n,m}\left(\frac{s}{u}\right)^n,
\end{equation}
and keep the terms in this expansion arising from $m\le n\le\kappa$. Lemma \ref{Error2}, below, shows that the discarded terms,
\begin{equation*}
\sum_{m=0}^{\min\left(\kappa,\ell\right)}\binom{\ell}{m}\frac{\Gamma(\kappa+1)}{\ell!}\frac{u^{\kappa}}{2\pi i}\int_{(c)}e^{s}R_{m}(u,s)\left(\log u-\log s-\gamma\right)^{\ell-m}\frac{ds}{s^{\kappa+1}},
\end{equation*}
where
\begin{equation}\label{E:Rdef}
R_{m}(u,s):=G(u,s)^m-\sum_{n=m}^{\kappa}E_{n,m}\left(\frac{s}{u}\right)^{n},
\end{equation}
are also bounded uniformly in $u\ge\ell$ by the stated error term in \eqref{KMain}.
Therefore, we have that
\begin{equation*}
K_{\ell}(u,\kappa)=\tilde{K}_{\ell}(u,\kappa)+O_{\kappa,\ell}\left(\frac{\left(1+\log u\right)^{\ell}}{u}\right),
\end{equation*}
where
\begin{equation*}
\tilde{K}_{\ell}(u,\kappa)=\sum_{m,n}\binom{\ell}{m}\frac{\Gamma(\kappa+1)}{\ell!}E_{n,m}\frac{u^{\kappa-n}}{2\pi i}\int_{(c)}e^{s}\left(\log u-\log s-\gamma\right)^{\ell-m}\frac{ds}{s^{\kappa-n+1}}.
\end{equation*}
The sum is over the pairs $m,n$ such that $0\le m\le \min\left(\kappa,\ell\right)$, $m\le n\le\kappa$.
Next, the binomial theorem is used to expand the integrand as
\begin{equation*}
\left(\log u-\log s-\gamma\right)^{\ell-m}=\sum_{r=0}^{\ell-m}\binom{\ell-m}{r}(-1)^{r}\left(\log s+\gamma\right)^{r}\log^{\ell-m-r}u,
\end{equation*}
so that
\begin{equation*}
\tilde{K}_{\ell}(u,\kappa)=\sum_{m,n,r}(-1)^{r}\binom{\ell}{m}\binom{\ell-m}{r}\frac{\Gamma(\kappa+1)r!}{\ell!}E_{n,m}C_{r,\kappa-n}u^{\kappa-n}\log^{\ell-m-r}u,
\end{equation*}
where the sum is over the triples $m,n,r$ such that $0\le m\le \min\left(\kappa,\ell\right)$, $m\le n\le\kappa$, and $0\le r\le\ell-m$, and $C_{r,\kappa-n}$ is given by the integral in \eqref{E1_begin}.\\ 
The constants $C_{r,\kappa}$ are generated by the series in \eqref{CDef} by Lemma \ref{T1_begin}. Therefore, we have the asymptotic formula
\begin{equation*}
K_{\ell}(u,\kappa)=\sum_{m=0}^{\min\left(\kappa,\ell\right)}\sum_{n=m}^{\kappa}\sum_{r=0}^{\ell-m}D_{m,n,r,\kappa,\ell}\ u^{\kappa-n}\log^{\ell-m-r}u+O_{\kappa,\ell}\left(\frac{(1+\log u)^{\ell}}{u}\right),
\end{equation*}
where
\begin{equation*}
D_{m,n,r,\kappa,\ell}=(-1)^{r}\binom{\ell}{m}\binom{\ell-m}{r}\frac{\Gamma(\kappa+1)r!}{\ell!}E_{n,m}C_{r,\kappa-n}.
\end{equation*}
Comparing this expression with \eqref{DDef}, we see that the proof of Theorem \ref{T_1} is complete.
\section{Lemmata}
This section is devoted to proving Lemma \ref{Error1} and \ref{Error2}. To begin, we will need to establish some simple bounds on $|G(u,s)|$, $|\frac{\partial}{\partial s}G(u,s)|$, $|R_{m}(u,s)|$, and $|\frac{\partial}{\partial s}R_{m}(u,s)|$ when $\Re s$ is fixed. In fact, we will assume henceforth that $\Re s=1$. 
\begin{lemma}\label{G_inequality}
Suppose that $\Re s=1$ and $u\ge 1$. Then we have
\begin{equation*}
  |G(u,s)| \ll
  \begin{cases}
  \frac{|s|}{u}& \text{if $|s|<u$,}\\
  1+\log\left(\frac{|s|}{u}\right)& \text{if $|s|>u$.}
  \end{cases}
\end{equation*}
\end{lemma}
\begin{proof}
This estimate follows immediately from the estimate of Soundararajan \cite[p.29]{Sound}
\begin{equation*}
  |G(u,s)|\ll \int_{0}^{1/u}\min\left(|s|,\frac{1}{t}\right)dt,
\end{equation*}
upon inspection of the integrand in \eqref{Gdef}. If $|s|<u$, then $t<\frac{1}{u}<\frac{1}{|s|}$, and so $\min\left(|s|,\frac{1}{t}\right)=|s|$, giving
\begin{equation*}
  |G(u,s)|\ll \int_{0}^{1/u}|s|dt=\frac{|s|}{u}.
\end{equation*}
On the other hand, if $|s|>u$, then we have $\frac{1}{|s|}<\frac{1}{u}$, and so
\begin{equation*}
  |G(u,s)|\ll \int_{0}^{1/|s|}\min\left(|s|,\frac{1}{t}\right)dt+\int_{1/|s|}^{1/u}\min\left(|s|,\frac{1}{t}\right)dt\ll 1+\log\left(\frac{|s|}{u}\right).
\end{equation*}
\end{proof}
\begin{lemma}\label{G'_inequality}
Suppose that $\Re s=1$ and that $u\ge 1$. Then we have
\begin{equation*}
  \left|\frac{\partial}{\partial s}G(u,s)\right|\ll
  \begin{cases}
  \frac{1}{u}& \text{if $|s|<u$,}\\
  \frac{1}{|s|}& \text{if $|s|>u$}.
  \end{cases}
\end{equation*}
\end{lemma}
\begin{proof}
This estimate is contained in Soundararajan's paper \cite[p.29]{Sound}. One uses the Fundamental Theorem of Calculus, so that
\begin{equation*}
  \frac{\partial}{\partial s}G(u,s)=\frac{1-e^{-s/u}}{s},
\end{equation*}
from which the estimate follows immediately.
\end{proof}
\begin{lemma}\label{L:R_ineq}
Suppose $0\le m\le\min(\kappa,\ell)$, $\Re s=1$, and $u\ge 1$. Then we have
\begin{equation*}
\left|R_{m}(u,s)\right|\ll_{\kappa,\ell}
\begin{cases}
\displaystyle\frac{\left|s\right|^{\kappa+1}}{u^{\kappa+1}},& \text{if $\left|s\right|<u$,}\\
\displaystyle\frac{\left|s\right|^{\kappa}}{u^{\kappa}},& \text{if $\left|s\right|>u$.}
\end{cases}
\end{equation*}
\end{lemma}
\begin{proof}
Since $R_m(u,s)=0$ when $m=0$, assume $m\ge1$. If $\left|s\right|<u$, then using \eqref{EGseries} and \eqref{E:Rdef},
\begin{equation}\label{Rseries}
R_{m}(u,s)=\sum_{n=\kappa+1}^{\infty}E_{n,m}\left(\frac{s}{u}\right)^{n}.
\end{equation}
We require a bound on $E_{n,m}$. From \eqref{Ealtexp}, it follows that for $m>1$,
\begin{equation}\label{Enm_bound}
  |n!E_{n,m}|\le\sum_{\substack{n_1,\ldots,n_m\ge1\\n_1+\cdots+n_m=n}}\binom{n}{n_1,\ldots,n_m}<m^n.
\end{equation}
This bound also applies to the case $m=1$, in which 
\begin{equation*}
  E_{n,1}=
  \begin{cases}
    \frac{(-1)^{n+1}}{n\ n!},& \text{if $n\ge1$,}\\
    0,& \text{if $n=0$.}
  \end{cases}  
\end{equation*}
Thus, we have that
\begin{equation*}
  |R_m(u,s)|\ll\frac{|s|^{\kappa+1}}{u^{\kappa+1}}\sum_{n=\kappa+1}^{\infty}|E_{n,m}|\ll\frac{|s|^{\kappa+1}}{u^{\kappa+1}}\sum_{n=\kappa+1}^{\infty}\frac{m^n}{n!}\ll_{\kappa,\ell}\frac{|s|^{\kappa+1}}{u^{\kappa+1}}.
\end{equation*}
On the other hand, if $\left|s\right|>u$, then using Lemma \ref{G_inequality} and \eqref{Enm_bound}, we have
\begin{equation*}
\left|R_{m}(u,s)\right|\ll|G(u,s)|^m+\frac{|s|^{\kappa}}{u^{\kappa}}\sum_{n=m}^{\kappa}|E_{n,m}|\ll_{\kappa,\ell} \left(1+\log\frac{|s|}{u}\right)^m+\frac{|s|^{\kappa}}{u^{\kappa}}\ll_{\kappa,\ell}\frac{|s|^{\kappa}}{u^{\kappa}}.
\end{equation*}
Here we have used the inequality $1+\log x\le x$ for $x\ge1$, and that $m\le\kappa$.
\end{proof}
\begin{lemma}\label{L:R'_ineq}
Suppose that $0\le m\le\min(\kappa,\ell)$, $\Re s=1$, and $u\ge 1$. Then we have
\begin{equation*}
\left|\frac{\partial}{\partial s}R_{m}(u,s)\right|\ll_{\kappa,\ell}
\begin{cases}
\displaystyle \frac{\left|s\right|^{\kappa}}{u^{\kappa+1}},& \text{if $\left|s\right|<u$,}\\
\displaystyle \frac{\left|s\right|^{\kappa-1}}{u^{\kappa}},& \text{if $\left|s\right|>u$.}
\end{cases}
\end{equation*}
\end{lemma}
\begin{proof}
Since $\frac{\partial}{\partial s}R_m(u,s)=0$ when $m=0$, assume $m\ge1$. Now, suppose that $\left|s\right|<u$. Using \eqref{Rseries},
\begin{align*}
\frac{\partial}{\partial s}R_{m}(u,s)&=\sum_{n=\kappa+1}^{\infty}n E_{n,m}\left(\frac{s}{u}\right)^{n-1}\left(\frac{1}{u}\right).
\end{align*}
Furthermore, applying our bound in \eqref{Enm_bound} yields
\begin{align*}
  \left|\frac{\partial}{\partial s}R_m(u,s)\right|&\ll\frac{|s|^{\kappa}}{u^{\kappa+1}}\sum_{n=\kappa+1}^{\infty} |nE_{n,m}|\ll\frac{|s|^{\kappa}}{u^{\kappa+1}}\sum_{n=\kappa+1}^{\infty}\frac{m^{n}}{(n-1)!}\ll_{\kappa,\ell} \frac{|s|^{\kappa}}{u^{\kappa+1}}.
\end{align*}
On the other hand, if $\left|s\right|>u$, then by \eqref{E:Rdef},
\begin{align*}
  \frac{\partial}{\partial s}R_m(u,s)&=m\frac{\partial}{\partial s}G(u,s)G(u,s)^{m-1}-\sum_{n=m}^{\kappa}n E_{n,m}\left(\frac{s}{u}\right)^{n-1}\left(\frac{1}{u}\right).
\end{align*}
Finally, using Lemma \ref{G_inequality} and Lemma \ref{G'_inequality}, we may conclude that
\begin{align*}
\left|\frac{\partial}{\partial s}R_{m}(u,s)\right|&\ll_{\kappa,\ell} \left|\frac{\partial}{\partial s}G(u,s)\right||G(u,s)|^{m-1}+\sum_{n=m}^{\kappa}|nE_{n,m}|\left(\frac{|s|}{u}\right)^{n-1}\frac{1}{u}\\
 &\ll_{\kappa,\ell} \frac{1}{|s|}\left(\frac{|s|}{u}\right)^{m-1}+\frac{|s|^{\kappa-1}}{u^{\kappa}}\sum_{n=m}^{\kappa}|nE_{n,m}|\ll_{\kappa,\ell}\frac{|s|^{\kappa-1}}{u^{\kappa}}.
\end{align*}
\end{proof}
Next, we will also make use of Lemma \ref{L1_L2_Lemma} and \ref{L:loglemma} extensively in the course of proving Lemma \ref{Error1} and \ref{Error2}. Their proofs are recorded here for completeness.
\begin{lemma}\label{L1_L2_Lemma}
  If $\ell\ge0$, $u\ge1$, and 
  \begin{align*}
  \Gamma_1&:=\{s\in\mathbb{C}: \Re s=1, \Im s>0, 1\le |s|<u\},\\
  \Gamma_2&:=\{s\in\mathbb{C}: \Re s=1, \Im s>0, |s|>u\}, 
  \end{align*}
  then we have
  \begin{align}\label{L1int}
    \int_{\Gamma_1} \frac{|ds|}{|s|}\ll 1+\log u,
  \end{align}
  and,
  \begin{equation}\label{L2int}
    \int_{\Gamma_2} \frac{\left(1+\log |s|\right)^\ell}{|s|^{2}}|ds|\ll_{\ell} \frac{(1+\log u)^\ell}{u}.
  \end{equation}
\end{lemma}
\begin{proof}
  To prove \eqref{L1int} and \eqref{L2int}, we let $s=1+it$ and observe that
  \begin{equation*}
    \int_{\Gamma_1}\frac{|ds|}{|s|}\ll\int_{0}^{\sqrt{u^2-1}}\frac{dt}{\sqrt{1+t^2}}\ll\int_{0}^{u}\frac{dt}{1+t}\ll 1+\log u,
  \end{equation*}
  and, using integration by parts, that
  \begin{equation*}
    \int_{\Gamma_2}\frac{\left(1+\log |s|\right)^\ell}{|s|^{2}}|ds|\ll_{\ell}\int_{u-1}^{\infty}\frac{\left(1+\log(1+t)\right)^\ell}{(1+t)^{2}}dt\ll_{\ell}\frac{\left(1+\log u\right)^\ell}{u}.
  \end{equation*}
 \end{proof}
\begin{lemma}\label{L:loglemma}
  If $\Re s=1$ and $u\ge1$, then we have
  \begin{equation*}
  |\log u-\log s-\gamma|\ll
  \begin{cases}
    1+\log u,& \text{if $|s|\le u$,}\\
    1+\log |s|,& \text{if $|s|\ge u$.}
  \end{cases}
  \end{equation*}
\end{lemma}
\begin{proof}
  Observe that since $\Re s=1$, we have $|\arg s|\le\frac{\pi}{2}$, and so
  \begin{equation*}
    |\log u-\log s-\gamma|\le \log u+\log |s|+ \frac{\pi}{2}+\gamma\ll \log u+\log |s|+ 1.
  \end{equation*}
  Both inequalities follow immediately.
\end{proof}
We are now ready to prove Lemma \ref{Error1} and Lemma \ref{Error2}.
\begin{lemma}\label{Error1}
If $\kappa<\ell$, $\kappa+1\le m\le\ell$, and $\Re s=1$, then for $u\ge\ell$,
\begin{equation*}
  \frac{u^{\kappa}}{2\pi i}\int_{(1)}e^{s}G(u,s)^m\left(\log u-\log s-\gamma\right)^{\ell-m}\frac{ds}{s^{\kappa+1}}\ll_{\kappa,\ell}\frac{(1+\log u)^{\ell}}{u}.
\end{equation*}
\end{lemma}
\begin{proof}
First, we integrate by parts and write
\begin{align*}
  -\frac{u^{\kappa}}{2\pi i}\int_{(1)}e^{s}\frac{\partial}{\partial s}\left(\frac{G(u,s)^{m}}{s^{\kappa+1}}\left(\log u-\log s-\gamma\right)^{\ell-m}\right)ds=\mathcal{I}_{1}+\mathcal{I}_{2}+\mathcal{I}_{3},
\end{align*}
where
\begin{align*}
  \mathcal{I}_{1}&=\frac{u^{\kappa}}{2\pi i}\int_{(1)}e^{s}\frac{(\log u-\log s-\gamma)^{\ell-m-1}}{s}\left(\frac{(\ell-m)G(u,s)^{m}}{s^{\kappa+1}}\right)ds,\\
  \mathcal{I}_{2}&=\frac{u^{\kappa}}{2\pi i}\int_{(1)}e^{s}\frac{(\log u-\log s-\gamma)^{\ell-m}}{s}\left(\frac{-smG(u,s)^{m-1}\frac{\partial}{\partial s}G(u,s)}{s^{\kappa+1}}\right)ds,\\
  \mathcal{I}_{3}&=\frac{u^{\kappa}}{2\pi i}\int_{(1)}e^{s}\frac{(\log u-\log s-\gamma)^{\ell-m}}{s}\left(\frac{(\kappa+1)G(u,s)^{m}}{s^{\kappa+1}}\right)ds.
\end{align*}
Next, we divide the integrals into two parts, one in which $1\le|s|<u$, and the other in which $|s|>u$. We need only consider $\mathcal{I}_{1}$, since the other two integrals are handled similarly. For brevity, we set $G(u,s)=G$.\\
Now, observe that $\mathcal{I}_{1}$ vanishes if $m=\ell$, so we may assume here that $\kappa+1\le m\le\ell-1$. Using Lemmas \ref{G_inequality}, \ref{L1_L2_Lemma}, and \ref{L:loglemma} gives
\begin{align*}
  \mathcal{I}_{1}&\ll_{\kappa,\ell} u^{\kappa}\left(\int_{\Gamma_1}\frac{\left(1+\log u\right)^{\ell-m-1}}{|s|}\frac{|G|^m}{|s|^{\kappa+1}}|ds|+\int_{\Gamma_2}\frac{\left(1+\log|s|\right)^{\ell-m-1}}{|s|}\frac{|G|^m}{|s|^{\kappa+1}}|ds|\right)\\
  &\ll_{\kappa,\ell} u^{\kappa}\int_{\Gamma_1}\frac{\left(1+\log u\right)^{\ell-m-1}}{|s|^{\kappa+2}}\frac{|s|^{\kappa+1}}{u^{\kappa+1}}|ds|+u^{\kappa}\int_{\Gamma_2}\frac{\left(1+\log|s|\right)^{\ell-1}}{|s|^{\kappa+2}}|ds|\\
  &\ll_{\kappa,\ell} \frac{(1+\log u)^{\ell-m-1}}{u} \int_{\Gamma_1}\frac{|ds|}{|s|}+\int_{\Gamma_2}\frac{\left(1+\log|s|\right)^{\ell-1}}{|s|^{2}}|ds|\\
  &\ll_{\kappa,\ell} \frac{(1+\log u)^{\ell}}{u},
\end{align*}
where $\Gamma_1$ and $\Gamma_2$ are as in Lemma \ref{L1_L2_Lemma}. The estimation of $\mathcal{I}_{2}$ additionally requires the use of Lemma \ref{G'_inequality}.
\end{proof}
\begin{lemma}\label{Error2}
If $0\le m\le\min\left(\kappa,\ell\right)$, and $\Re s=1$, then for $u\ge\ell$,
\begin{equation*}
u^{\kappa}\frac{1}{2\pi i}\int_{(1)}e^{s}R_{m}(u,s)\left(\log u-\log s-\gamma\right)^{\ell-m}\frac{ds}{s^{\kappa+1}}\ll_{\kappa,\ell}\frac{(1+\log u)^{\ell}}{u}.
\end{equation*}
\end{lemma}
\begin{proof}
We will proceed as in Lemma \ref{Error1}, integrating by parts and writing
\begin{align*}
-\frac{u^{\kappa}}{2\pi i}\int_{(1)}e^{s}\frac{\partial}{\partial s}\left(R_{m}(u,s)\left(\log u-\log s-\gamma\right)^{\ell-m}\frac{1}{s^{\kappa+1}}\right)ds=\mathcal{I}_{1}+\mathcal{I}_{2}+\mathcal{I}_{3},
\end{align*}
where
\begin{align*}
  \mathcal{I}_{1}&=\frac{u^{\kappa}}{2\pi i}\int_{(1)}e^{s}\frac{\left(\log u-\log s-\gamma\right)^{\ell-m-1}}{s}\left(\frac{(\ell-m)R_{m}(u,s)}{s^{\kappa+1}}\right)ds\\
  \mathcal{I}_{2}&=\frac{u^{\kappa}}{2\pi i}\int_{(1)}e^{s}\frac{\left(\log u-\log s-\gamma\right)^{\ell-m}}{s}\left(\frac{-s \frac{\partial}{\partial s}R_{m}(u,s)}{s^{\kappa+1}}\right)ds\\
  \mathcal{I}_{3}&=\frac{u^{\kappa}}{2\pi i}\int_{(1)}e^{s}\frac{\left(\log u-\log s-\gamma\right)^{\ell-m}}{s}\left(\frac{(\kappa+1)R_{m}(u,s)}{s^{\kappa+1}}\right)ds.
\end{align*}
We again divide the integrals into two parts, one in which $1\le|s|<u$, and the other in which $|s|>u$. We consider only $\mathcal{I}_{1}$, since the other two integrals are handled similarly. For brevity, we set $R_m(u,s)=R$.\\
Assume that $m\ge1$ since $R_{m}(u,s)=0$ if $m=0$. Also, $\mathcal{I}_{1}$ vanishes if $m=\ell$, so we may assume that $1\le m\le\min\left(\kappa,\ell-1\right)$. Using Lemmas \ref{L:R_ineq}, \ref{L1_L2_Lemma}, and \ref{L:loglemma} gives
\begin{align*}
  \mathcal{I}_{1}&\ll_{\kappa,\ell} u^{\kappa}\left(\int_{\Gamma_1}\frac{\left(1+\log u\right)^{\ell-m-1}}{|s|}\frac{|R|}{|s|^{\kappa+1}}|ds|+\int_{\Gamma_2}\frac{\left(1+\log|s|\right)^{\ell-m-1}}{|s|}\frac{|R|}{|s|^{\kappa+1}}|ds|\right)\\
  &\ll_{\kappa,\ell} u^{\kappa}\int_{\Gamma_1}\frac{(1+\log u)^{\ell-m-1}}{|s|^{\kappa+2}}\frac{|s|^{\kappa+1}}{u^{\kappa+1}}|ds|+u^{\kappa}\int_{\Gamma_2}\frac{\left(1+\log|s|\right)^{\ell-m-1}}{|s|^{\kappa+2}}\frac{|s|^{\kappa}}{u^{\kappa}}|ds|\\
  &\ll_{\kappa,\ell} \frac{(1+\log u)^{\ell-m-1}}{u}\int_{\Gamma_1}\frac{|ds|}{|s|}+\int_{\Gamma_2}\frac{(1+\log|s|)^{\ell-m-1}}{|s|^{2}}|ds|\\
  &\ll_{\kappa,\ell} \frac{(1+\log u)^{\ell}}{u},
\end{align*}
where $\Gamma_1$ and $\Gamma_2$ are as in Lemma \ref{L1_L2_Lemma}. The estimation of $\mathcal{I}_{2}$ additionally requires the use of Lemma \ref{L:R'_ineq}.
\end{proof}
\section{Conclusion}
It is possible to generalize Theorem \ref{T_1} to all real $\kappa>-1$, at the expense of a slightly worse error term. In addition, the expansion is relevant to a number of functions satisfying the differential-delay equation $(u^{a}p(u))'=-bu^{a-1}p(u-1)$, including the Ankeny-Onishi-Selberg function, $j_{\kappa}(u)=\sigma_{\kappa}(2u)$, featured in the Selberg sieve. Although interest has been expressed in the arithmetic nature of the coefficients of the expansion, it may also be useful for computational purposes, given the great deal of uniformity in $u$.
\section*{Acknowledgement}
The author would like to thank the Mathematics Research Communities program for the opportunity to participate in the workshop on the Pretentious View of Analytic Number Theory at Snowbird, Utah, which was an inspiration for the preparation of the present paper.
\bibliographystyle{abbrv}
\bibliography{\jobname} 
\end{document}